\documentclass[12pt]{amsart}

\usepackage{amsmath,amsthm,amsfonts,verbatim,dsfont}
\usepackage{xypic}
\xyoption{all} \xyoption{matrix}

\def\wt{\widetilde}
\def\ra{\overline}
\def\m{\mathfrak{m}}
\def\Z{\mathcal{Z}}

\def\Aut{\mathrm{Aut}}

\def\id{\mathrm{id}}
\def\ima{\mathrm{Im}}
\def\op{\mathrm{op}}

\def\HC{\mathrm{HC}}
\def\HN{\mathrm{HN}}
\def\HP{\mathrm{HP}}
\def\HH{\mathrm{HH}}
\def\H{\mathrm{H}}

\DeclareMathOperator{\BC}{\mathsf{BC}} \DeclareMathOperator{\BN}{\mathsf{BN}}
\DeclareMathOperator{\BP}{\mathsf{BP}}

\def\x{\mathbf{x}}
\def\a{\mathbf{a}}

\def\I{\mathds{I}}
\def\J{\mathds{J}}

\def\C{\mathbb{C}}

\def\ot{\otimes}
\def\cl#1{{\langle #1\rangle}}

\theoremstyle{plain}
\newtheorem{teo}{Theorem}[section]

\newtheorem{coro}[teo]{Corollary}
\newtheorem{prop}[teo]{Proposition}
\theoremstyle{definition}
\newtheorem{ex}[teo]{Example}
\newtheorem{rem}[teo]{Remark}

\title[The cyclic homology of monogenic extensions]{The cyclic homology of monogenic
extensions in the noncommutative setting}

\begin{document}

\author{Graciela Carboni}
\address{C\'\i clo B\'asico Com\'un\\ Pabell\'on 3 - Ciudad Universitaria\\ (1428) Buenos
Aires, Argentina.} \curraddr{} \email{gcarbo@dm.uba.ar}
\thanks{Supported by UBACYT X0294}

\author{Jorge A. Guccione}
\address{Departamento de Matem\'atica\\ Facultad de Ciencias Exactas y Naturales\\
Pabell\'on 1 - Ciudad Universitaria\\ (1428) Buenos Aires, Argentina.}\curraddr{}
\email{vander@dm.uba.ar}
\thanks{Supported by PICT 12330, UBACYT X0294 and CONICET}

\author{Juan J. Guccione}
\address{Departamento de Matem\'atica\\ Facultad de Ciencias Exactas y Naturales\\
Pabell\'on 1 - Ciudad Universitaria\\ (1428) Buenos Aires, Argentina.} \curraddr{}
\email{jjgucci@dm.uba.ar}
\thanks{Supported by PICT 12330, UBACYT X0294 and CONICET}

\date{}


\dedicatory{}


\begin{abstract} We study the Hochschild and cyclic homologies of noncommutative monogenic
extensions. As an aplication we compute the Hochschild and cyclic homologies of the rank~$1$ Hopf
algebras introduced in \cite{K-R}.
\end{abstract}

\maketitle

\section*{Introduction}
Let $k$ be a commutative ring with $1$. A monogenic extension of $k$ is a $k$-algebra
$k[x]/\langle f \rangle$, where $f\in k[x]$ is a monic polynomial. In \cite{F-G-G} this concept
was generalized to the non-commutative setting.  Examples are the rank~$1$ Hopf algebras in
characteristic zero, recently introduced in \cite{K-R}. In the paper \cite{F-G-G}, mentioned
above, it was compute the Hochschild cohomology ring of these extensions. In the present paper we
study their Hochschild and cyclic homology groups. The main result, obtained by us, is that, for
any monogenic extension $A$ of a $k$-algebra $K$, there exists a small mixed complex
$(C^S_*(A),d_*,D_*)$, giving the Hochschild and cyclic homology of $A$ relative to $K$. When $K$
is a separable $k$-algebra this gives the absolute Hochschild and cyclic homology groups. The
mixed complex $(C^S(A),d_*,D_*)$ is enough simple to allow us to compute the homology of each
rank~$1$ Hopf algebras.

\smallskip

The paper is organized as follows: In Section~1 we give some preliminaries we need. In particular
we recall the simple $\Upsilon$-projective resolution of a monogenic extension $A/K$ (where
$\Upsilon$ is the family of all $A$-bimodule epimorphisms which split as $K$-bimodule map), build
in \cite{F-G-G}. In Section~2 we use the mentioned above resolution to build a complex
$C^S(A,M)=(C^S_*(A,M),d_*)$ giving the relative to $K$ Hochschild homology of $A$ relative to $M$,
for each $A$-bimodule $M$ (symmetric over $k$). Then, we obtain explicit computations under
suitable hypothesis. In Section~3 we prove that $C^S(A,A)$ is the Hochschild complex of a mixed
complex. This generalizes the main result of \cite{Bach}. Then, we use this result to compute the
cyclic homology of $A$ in several cases, including the rank~$1$ Hopf algebras.

\section{Preliminaries}
In this section we recall some well known definitions and results, and we fix some notations
that we will use in the rest of the paper.
\subsection{A simple resolution for a noncommutative monogenic extension}
In this subsection we recall some definitions and results pro\-ved in \cite{F-G-G}. Let $k$ be a
commutative ring, $K$ an associative $k$-algebra, which we do not assume to be commutative, and
$\alpha$ a $k$-algebra endomorphism of $K$. Consider the Ore extension $B=K[x,\alpha]$, that is
the algebra generated by $K$ and $x$ subject to the relations
\[
x\lambda=\alpha(\lambda)x\quad\text{for all } \lambda\in K.
\]
Let $f=x^n+\sum_{i=1}^{n}\lambda_ix^{n-i}$ be a monic polynomial of degree $n\ge 2$, where each
coefficient $\lambda_i\in K$ satisfies $\alpha(\lambda_i)=\lambda_i$ and $\lambda_i\lambda =
\alpha^i(\lambda)\lambda_i$ for every $\lambda\in K$. Sometimes we will write $f = \sum_{i=0}^n
\lambda_ix^{n-i}$, assuming that $\lambda_0=1$. The monogenic extension of $K$ associated with
$\alpha$ and $f$ is the quotient $A=B/\cl{f}$. It is easy to see that $\{1,x,\dots,x^{n-1}\}$
is a left $K$-basis of the algebra $A$. Moreover, given $P\in B$, there exist unique $\ra{P}$
and $\stackrel{\dots}{P}$ in $B$ such that
\[
P=\ra{P}f+\stackrel{\dots}{P}\quad\text{and}\quad \stackrel{\dots}{P}=0  \text{ or } \deg
\stackrel{\dots}{P}<n.
\]
In this paper, unadorned tensor product $\ot$ means $\ot_K$ and all the maps are $k$-linear. Given
a $K$-bimodule $M$, we let $M\ot$ denote the quotient $M/[M,K]$, where $[M,K]$ is the $k$-module
generated by the commutators $m\lambda - \lambda m$ with $\lambda\in K$ and $m\in M$. Let
$A^2_{\alpha^r} =A_{\alpha^r}\ot A$, where $A_{\alpha^r}$ is $A$ endowed with the regular left
$A$-module structure and with the right $K$-module structure twisted by $\alpha^r$, that is, if
$a\in A_{\alpha^r}$ and $\lambda \in K$, then $a\cdot \lambda = a \alpha^r(\lambda)$. We recall
that
\[
\frac{T}{Tx}: B\to A^2_{\alpha}
\]
is the unique $K$-derivation such that $\frac{Tx}{Tx}=1\ot 1$. Notice that
\[
\frac{Tx^i}{Tx} = \sum_{\ell=0}^{i-1}x^{\ell}\ot x^{i-\ell-1}.
\]
Let $\Upsilon$ be the family of all $A$-bimodule epimorphisms which split as $K$-bimodule maps.

\begin{teo}{\rm(\cite[Theorem 2.1]{F-G-G})}\label{teorema 1.1} The complex
\[
C_S'(A) = \xymatrix@C-8pt{\cdots\ar[r]& A^2_{\alpha^{2n+1}} \ar[r]^-{d'_5}& A^2_{\alpha^{2n}}
\ar[r]^-{d'_4}& A^2_{\alpha^{n+1}}\ar[r]^-{d'_3}& A^2_{\alpha^n}\ar[r]^-{d'_2}& A^2_{\alpha}
\ar[r]^-{d'_1}& A^2},
\]
where
\[
d'_{2m+1}\colon  A^2_{\alpha^{mn+1}}\to A^2_{\alpha^{mn}} \quad \text{and} \quad d'_{2m}\colon
A^2_{\alpha^{mn}}\to A^2_{\alpha^{(m-1)n+1}},
\]
are defined by
\begin{align*}
& d'_{2m+1}(1\ot 1) = x\ot 1-1\ot x,\\
& d'_{2m}( 1\ot 1) = \frac{Tf}{Tx}= \sum_{i=1}^{n} \lambda_{n-i}\sum_{\ell=0}^{i-1}x^{\ell}\ot
x^{i-\ell-1},
\end{align*}
is a $\Upsilon$-projective resolution of $A$.
\end{teo}

Let $(A\ot \ra{A}^{\ot^*}\ot A,b')$ be the canonical Hochschild resolution relative to $K$. Here
$\ra{A} = A/K$.

\begin{teo}{\rm(\cite[Proposition 2.2 and Theorem 2.3]{F-G-G})}\label{teorema 1.2} There are comparison maps
$$
\phi'_*\colon C_S'(A)\to (A\ot \ra{A}^{\ot^*}\ot A,b')\quad\text{and}\!\quad\! \psi'_*\colon
(A\ot \ra{A}^{\ot^*}\ot A,b')\to C_S'(A),
$$
which are inverse one of each other up to homotopy. These maps are given by
\begin{align*}
& \phi'_0(1\ot 1)= 1\ot 1,\\
& \phi'_1(1\ot 1)= 1\ot x\ot 1,\\
&\phi'_{2m}(1\ot 1) = \sum_{\mathbf{i}\in \I_m} \boldsymbol{\lambda}_{\mathbf{n}- \mathbf{i}}
\sum_{\boldsymbol{\ell}\in \J_{\mathbf{i}}}
x^{|\mathbf{i}-\boldsymbol{\ell}|-m} \ot \wt{\x}^{\boldsymbol{\ell}_{m,1}}\ot 1,\\
&\phi'_{2m+1}(1\ot 1) = \sum_{\mathbf{i}\in \I_m} \boldsymbol{\lambda}_{\mathbf{n}- \mathbf{i}}
\sum_{\boldsymbol{\ell}\in \J_{\mathbf{i}}}
x^{|\mathbf{i}-\boldsymbol{\ell}|-m} \ot \wt{\x}^{\boldsymbol{\ell}_{m,1}}\ot x\ot 1,\\
&\psi'_{2m}(\hat{\x}^{\mathbf{i}_{1,2m}})= \ra{x^{i_1+i_2}}\, \ra{x^{i_3+i_4}}\cdots
\ra{x^{i_{2m-1}+i_{2m}}}\ot 1,\\
&\psi'_{2m+1}(\hat{\x}^{\mathbf{i}_{1,2m+1}})= \ra{x^{i_1+i_2}}\, \ra{x^{i_3+i_4}} \cdots
\ra{x^{i_{2m-1}+i_{2m}}}\frac{T (x^{i_{2m+1}})}{Tx} ,
\end{align*}
where

\begin{itemize}

\smallskip

\item  $\I_m=\{(i_1,\dots,i_m)\in \mathbb{Z}^m: 1\le i_j\le n  \text{ for all $j$}\}$,

\smallskip

\item $\J_{\mathbf{i}}=\{(l_1,\dots,l_m)\in \mathbb{Z}^m: 1\le l_j<i_j \text{ for all
$j$}\}$,

\smallskip

\item $\boldsymbol{\lambda}_{\mathbf{n}- \mathbf{i}} =\lambda_{n-i_1}\cdots \lambda_{n-i_m}$,

\smallskip

\item $\wt{\x}^{\boldsymbol{\ell}_{m,1}} = x\ot x^{\ell_m} \ot \cdots\ot x \ot x^{\ell_1}$,

\smallskip

\item $|\mathbf{i}-\boldsymbol{\ell}|=\sum_{j=1}^m (i_j -\ell_j)$.

\smallskip

\item $\hat{\x}^{\mathbf{i}_{1r}} = 1\ot x^{i_1}\ot \cdots\ot x^{i_r}\ot 1$,

\end{itemize}
\end{teo}

\begin{prop}\label{prop 1.3} $\psi'_*\phi'_*= \id$ and a homotopy $\omega'_{*+1}$ from $\phi'_*
\psi'_*$ to $\id$ is recursively defined by $\omega'_1=0$ and
\begin{align*}
\omega'_{r+1}(\x\ot 1) & = (-1)^r(\phi'_r\psi'_r-\id-\omega'_rb'_r)(\x\ot1)\ot 1\\
& = (-1)^r \phi'_r\psi'_r(\x\ot 1)\ot 1 - \omega'_r(\x)\ot 1,
\end{align*}
for $\x\in A \ot \ra{A}^{\ot^r}$.

\end{prop}

\begin{proof} The equality $\psi'_*\phi'_*= \id$ follows immediately from the definitions.
For the second assertion see \cite[Proposition 1.2.1]{G-G}.
\end{proof}

\subsection{Mixed complexes}
In this subsection we recall briefly the notion of mixed complex. For more details about this
concept we refer to \cite{K} and \cite{B}.

\smallskip

A mixed complex $(X,b,B)$ is a graded $C$-module $(X_r)_{r\ge 0}$, endowed with morphisms
$b\colon X_r\to X_{r-1}$ and $B\colon X_r\to X_{r+1}$, such that
$$
b b = 0,\quad B B = 0\quad\text{and}\quad B  b + b B = 0.
$$
A morphism of mixed complexes $f\colon (X,b,B)\to (Y,d,D)$ is a family $f_r\colon X_r\to Y_r$ of
maps, such that $d f = f b$ and $D f= f B$. A mixed complex $\mathcal{X} = (X,b,B)$ determines a
double complex
\[
\xymatrix{\\\\ \BP(\mathcal{X})=}\qquad
\xymatrix{
& \vdots \dto^-{b} &\vdots \dto^-{b}& \vdots \dto^-{b}& \vdots \dto^-{b}\\
\dots & X_3 \lto_-{B}\dto^-{b} & X_2 \lto_-{B}\dto^-{b} & X_1 \lto_-{B}& X_0 \lto_-{B}\\
\dots & X_2 \lto_-{B}\dto^-{b} & X_1 \lto_-{B}\dto^-{b} & X_0 \lto_-{B}\\
\dots & X_1 \lto_-{B}\dto^-{b} & X_0 \lto_-{B}\\
\dots & X_0 \lto_-{B}}
\]
By deleting the positively numbered columns we obtain a subcomplex $\BN(\mathcal{X})$ of
$\BP(\mathcal{X})$. Let $\BN'(\mathcal{X})$ be the kernel of the canonical surjection from
$\BN(\mathcal{X})$ to $(X,b)$. The quotient double complex $\BP(\mathcal{X}) /\BN'(\mathcal{X})$
is denoted by $\BC(\mathcal{X})$. The homologies $\HC_*(\mathcal{X})$, $\HN_*(\mathcal{X})$ and
$\HP_*(\mathcal{X})$, of the total complexes of $\BC(\mathcal{X})$, $\BN(\mathcal{X})$ and
$\BP(\mathcal{X})$ respectively, are called the cyclic, negative and periodic homologies of
$\mathcal{X}$. The homology $\HH_*(\mathcal{X})$, of $(X,b)$, is called the Hochschild homology of
$\mathcal{X}$. Finally, it is clear that a morphism $f\colon \mathcal{X}\to \mathcal{Y}$ of mixed
complexes induces a morphism from the double complex $\BP(\mathcal{X})$ to the double complex
$\BP(\mathcal{Y})$. Let $A$ be a noncommutative monogenic extension of $K$. The normalized mixed
complex of $A$ relative to $K$ is $(A\ot \ra{A}^{\ot^*}\ot,b,B)$, where $b$ is the canonical
Hochschild boundary map and
$$
B([a_0\ot\cdots\ot a_r]) = \sum_{i=0}^r (-1)^{ir} [1\ot a_i\ot\cdots\ot a_r\ot a_0\ot\cdots\ot
a_{i-1}],
$$
in which $[a_0\ot\cdots\ot a_r]$ denotes the class of $a_0\ot\cdots\ot a_r$ in $A\ot
\ra{A}^{\ot^r}\ot$. The cyclic, negative, periodic and Hochschild homologies $\HC^K_*(A)$,
$\HN^K_*(A)$, $\HP^K_*(A)$ and $\HH^K_*(A)$ of $A$ are the respective homologies of
$(A\ot\ra{A}^{\ot^*}\ot,b,B)$.

\subsection{The perturbation lemma}
Next we recall the perturbation le\-mma. We give the more general version introduced in \cite{C}.

\smallskip

A homotopy equivalence data
\begin{equation}
\xymatrix{(Y,\partial)\ar@<-1ex>[r]_-{i} & (X,d) \ar@<-1ex>[l]_-{p}}, \quad h\colon X_*\to
X_{*+1},\label{eq1}
\end{equation}
consists of the following:

\begin{enumerate}

\smallskip

\item Chain complexes $(Y,\partial)$, $(X,d)$ and quasi-isomorphisms $i$ and $p$ between them,

\smallskip

\item A homotopy $h$ from $ip$ to $\id$.
\end{enumerate}

\smallskip

\noindent A perturbation~$\delta$ of~\eqref{eq1} is a map $\delta\colon X_*\to X_{*-1}$ such
that $(d+\delta)^2 = 0$. We call it small if $\id - \delta h$ is invertible. In this case we
write
$$
\Delta = (\id - \delta h)^{-1} \delta
$$
and we consider
\begin{equation}
\xymatrix{(Y,\partial^1)\ar@<-1ex>[r]_-{i^1} & (X,d+\delta)\ar@<-1ex>[l]_-{p^1}}, \quad
h^1\colon X_*\to X_{*+1},\label{eq2}
\end{equation}
with
$$
\partial^1 = \partial + p \Delta i,\quad i^1 = i + h \Delta i,\quad
p^1 = p + p \Delta h,\quad h^1 = h + h \Delta h.
$$
A deformation retract is a homotopy equivalence data such that $p i = \id$. A deformation
retract is called special if $h i = 0$, $p h = 0$ and $h h = 0$.

\smallskip

In the case considered in this paper the map $\delta h$ is locally nilpotent, and so $(\id -
\delta h)^{-1} = \sum_{j=0}^{\infty} (\delta h)^j$.

\begin{teo} (\cite{C}) If $\delta$ is a small perturbation of the homotopy equivalence
data~\eqref{eq1}, then the perturbed data~\eqref{eq2} is a homotopy equivalence. Moreover, if
\eqref{eq1} is a special deformation retract, then~\eqref{eq2} is also.\label{teorema 1.4}
\end{teo}

\section{Hochschild homology of $A$}
Given an $A$-bimodule $M$, we let $[M,K]_{\alpha^j}$ denote the $k$-submodule of $M$ generated by
the twisted commutators $[m,\lambda]_{\alpha^j} = m\alpha^j(\lambda) - \lambda m$. As usual, we
let $A^e$ and $\H^K_*(A,M)$ denote the enveloping algebra $A\ot_k A^{\op}$ of $A$ and the relative
to $K$ Hochschild homology of $A$ with coefficients in $M$, respectively. When $M = A$ we will
write $\HH^K_*(A)$ instead of $\H^K_*(A,A)$.

In the following proposition we use the same notations as in Theorem~\ref{teorema 1.1}.

\begin{teo}\label{teorema 2.1} Let $M$ be an $A$-bimodule. The following facts hold:
\begin{enumerate}

\smallskip

\item The chain complex
\[
\spreaddiagramcolumns{-0.3pc} \quad\qquad C^S(A,M)= \xymatrix@C-8pt{\cdots  \ar[r]^-{d_4}&
\frac{M}{[M,K]_{\alpha^{n+1}}} \ar[r]^-{d_3}& \frac{M}{[M,K]_{\alpha^n}} \ar[r]^-{d_2}&
\frac{M}{[M,K]_{\alpha}} \ar[r]^-{d_1}& \frac{M}{[M,K]_{\alpha^0}}},
\]
where the boundaries are the maps defined by
\begin{align*}
& d_{2m+1}([\m]) = [\m x-x\m],\\
&d_{2m}([\m]) =\sum_{i=1}^{n}\sum_{\ell=0}^{i-1} [ \lambda_{n-i}x^{i-\ell-1}\m x^{\ell}],
\end{align*}
in which $[\m]$ denotes the class of $\m\in M$ in $\frac{M}{[M,K]_{\alpha^{mn+1}}}$ and
$\frac{M}{[M,K]_{\alpha^{mn}}}$ respectively, computes $\H^K_*(A,M)$.

\medskip

\item The maps
\begin{align*}
&\phi_*\colon C^S(A,M)\to(M\ot \ra{A}^{\ot^*}\ot,b_*),\\
&\psi_*\colon (M\ot \ra{A}^{\ot^*}\ot,b_*)\to C^S(A,M),
\end{align*}
defined by
\begin{align*}
\qquad\quad & \phi_0([\m])= [\m],\\
& \phi_1([\m])= [\m\ot x],\\
&\phi_{2m}([\m]) = \sum_{\mathbf{i}\in\I_m}  \sum_{\boldsymbol{\ell}\in \J_{\mathbf i}}
[\boldsymbol{\lambda}_{\mathbf{n}- \mathbf{i}}\m x^{|\mathbf{i}-
\boldsymbol{\ell}|-m}\ot \wt{\x}^{\boldsymbol{\ell}_{m,1}}],\\
&\phi_{2m+1}([\m])=\sum_{\mathbf{i}\in\I_m} \sum_{\boldsymbol{\ell}\in\J_{\mathbf i}}
[\boldsymbol{\lambda}_{\mathbf{n}- \mathbf{i}}\m
x^{|\mathbf{i}-\boldsymbol{\ell}|-m} \ot \wt{\x}^{\boldsymbol{\ell}_{m,1}}\ot x],\\
&\psi_{2m}([\m\ot\x^{\mathbf{i}_{1,2m}}])=[\m\ra{x^{i_1+i_2}}\cdots\ra{x^{i_{2m-1}+i_{2m}}}],\\
&\psi_{2m+1}([\m\ot\x^{\mathbf{i}_{1,2m+1}}])\!=\!\!\sum_{\ell=0}^{i_{2m+1}-1}
\!\!\![x^{i_{2m+1}-\ell-1}\m \ra{x^{i_1\!+i_2}} \cdots \ra{x^{i_{2m-1}\!+i_{2m}}}x^{\ell}],
\end{align*}
where $[\m\ot \x^{\mathbf{i}_{1r}}]$ denotes the class of $\m\ot \x^{\mathbf{i}_{1r}}$ in $M\ot
\ra{A}^{\ot^r}\ot$, are chain morphisms which are inverse one of each other up to homotopy.

\smallskip

\item Let
$$
\qquad\beta\colon M \ot_{A^e} A\ot \ra{A}^{\ot^{r+1}}\ot A\to M\ot\ra{A}^{\ot^{r+1}}\ot
$$
be the map defined by
$$
\beta_{r+1}(m\ot \x_{0,r+2}) = [x_{r+2}mx_0\ot \x_{1,r+1}].
$$
The composition $\psi_*\phi_*$ gives the identity map, and the family of maps
$$
\qquad\omega_{*+1}\colon M\ot \ra{A}^{\ot^*}\ot \to M\ot \ra{A}^{\ot^{*+1}}\ot,
$$
defined by
$$
\omega_{r+1}([m\ot \x]) = \beta \bigl(m\ot_{A^e}\omega'_{r+1}(1\ot\x\ot 1)\bigr),
$$
is an homotopy from $\phi_*\psi_*$ to the identity map.
\end{enumerate}
\end{teo}

\begin{proof} For the first item, apply the functor $M\ot_{A^e}-$ to the resolution
$C_S'(A)$, and use the identification
\[
\xymatrix@R-4ex{M\ot_{A^e} A^2_{\alpha^j}\ar[r]^-{\cong}& \frac{M}{[M,K]_{\alpha^j}}
\\
\m\ot (a\ot b)\ar@{|->}[r] & [b\m a].}
\]
For instance

\begin{align*}
d_{2m}([\m]) & = \sum_{i=1}^n  \sum_{\ell = 0}^{i-1} [x^{i-\ell-1}\m\lambda_{n-i}x^{\ell}]\\
& = \sum_{i=1}^n  \sum_{\ell = 0}^{i-1} [x^{i-\ell-1}\m x^{\ell}\lambda_{n-i}]\\
& = \sum_{i=1}^n  \sum_{\ell = 0}^{i-1} [\lambda_{n-i} x^{i-\ell-1}\m x^{\ell}].
\end{align*}
Let $\psi_*$ and $\phi_*$ be the morphism induced by the comparison maps $\psi'_*$ and $\phi'_*$.
The second and third item follow easily from this definition in the same manner.
\end{proof}

When $M = A$ we will write $C^S(A)$ instead of $C^S(A,M)$.

\subsection{Explicit computations}
The aim of this subsection is to compute the Hochschild homology of $A$ with coefficients in $A$,
under suitable hypothesis.

\begin{teo}\label{teorema 2.2} Let $C^S_r(A)$ denote the $r$-th module of $C^S(A)$. If
there exists $\breve{\lambda}\in \Z(K)$ such that
\begin{itemize}

\smallskip

\item $\alpha^n(\breve{\lambda})=\breve{\lambda}$,

\smallskip

\item $\breve{\lambda}-\alpha^i(\breve{\lambda})$ is invertible in $K$ for $1\le i<n$,

\end{itemize}
then
\[
C^S_r(A) = \begin{cases} \frac{K}{[K,K]_{\alpha^{mn}}}&\text{if $r=2m$,}\\
\frac{K}{[K,K]_{\alpha^{(m+1)n}}}x^{n-1}& \text{if $r=2m+1$.}\end{cases}
\]
\end{teo}

\begin{proof} By item~(1) of Theorem~\ref{teorema 2.1} we know that
$$
C^S_r(A) = \begin{cases} \frac{A}{[A,K]_{\alpha^{mn}}}&\text{if $r=2m$,}\\
\frac{A}{[A,K]_{\alpha^{mn+1}}} & \text{if $r=2m+1$.}\end{cases}
$$
Moreover
$$
[a,\lambda]_{\alpha^r}= \sum_{i=0}^{n-1}[\lambda'_i,\lambda] _{\alpha^{r+i}}x^i
$$
for each $a=\sum_{i=0}^{n-1} \lambda'_ix^i\in A$ and $\lambda\in K$. Hence, it will be
sufficient to check that if $i$ is not congruent to $0$ module $n$, then
$[K,K]_{\alpha^{mn+i}}=K$. But this follows immediately from the fact that
$\alpha^i(\breve{\lambda}) - \breve{\lambda}$ is invertible if $i$ is not congruent to $0$
module $n$ and
\[
[\lambda',\breve{\lambda}]_{\alpha^{mn+i}}= \lambda'\alpha^{mn+i}(\breve{\lambda}) -
\breve{\lambda}\lambda' = \lambda'(\alpha^i(\breve{\lambda})-\breve{\lambda}),
\]
since $\breve{\lambda}\in \Z(K)$ and $\alpha^n(\breve{\lambda}) = \breve{\lambda}$.
\end{proof}

\begin{teo}\label{teorema 2.3} Under the hypothesis of Theorem~\ref{teorema 2.2}, the
boundary maps of $C^S(A)$ are given by
\begin{align*}
& d_{2m+1}([\lambda]x^{n-1} )= [(\alpha(\lambda)-\lambda)\lambda_n],\\
& d_{2m+2}([\lambda])= \Biggl[\sum_{\ell=0}^{n-1} \alpha^{\ell}(\lambda) \Biggr]x^{n-1},
\end{align*}
for all $m\ge 0$. Consequently, if $\lambda_n=0$, then the odd boundary maps are zero.
\end{teo}

\begin{proof} By item~(1) of Theorem~\ref{teorema 2.1},
\[
d_{2m+1}([\lambda]x^{n-1} )=[\lambda x^n-x\lambda x^{n-1}]=[(\lambda-\alpha(\lambda))x^n] =
[(\alpha(\lambda)-\lambda)\lambda_n],
\]
where the last equality follows from Theorem~\ref{teorema 2.2}. Again by item~(1) of
Theorem~\ref{teorema 2.1},
\begin{align*}
d_{2m+2}([\lambda]) &=\sum_{i=1}^{n}\sum_{\ell=0}^{i-1} [\lambda_{n-i}x^{i-\ell-1}\lambda
x^{\ell}]\\
&=\sum_{i=1}^{n}\sum_{\ell=0}^{i-1}[\lambda_{n-i}\alpha^{i-l-1}(\lambda)x^{i-1}]\\
&= \Biggl[\sum_{\ell=0}^{n-1} \alpha^{n-\ell-1}(\lambda) \Biggr]x^{n-1},
\end{align*}
where the last equality follows again from Theorem~\ref{teorema 2.2}.
\end{proof}

Theorem~\ref{teorema 2.3} implies that $\lambda\lambda_n-\alpha^n(\lambda)\lambda_n\in
[K,K]_{\alpha^{mn}}$ for all $\lambda\in K$ and $m\ge 0$. Indeed, this can be proved directly from
the hypothesis at the beginning of this paper and then it is true with full generality. In fact,
$$
\lambda\lambda_n-\alpha^n(\lambda)\lambda_n = \lambda\lambda_n-\lambda_n\lambda =
\lambda\alpha^{mn}(\lambda_n)-\lambda_n\lambda.
$$

\begin{coro}\label{corolario 2.4} Under the hypothesis of Theorem~\ref{teorema 2.2},
\begin{align*}
& \HH^K_0(A)=\frac{K}{[K,K]+\ima(\alpha-\id)\lambda_n},\\
& \HH^K_{2m+1}(A)=\frac{\left\{\lambda\in K:(\alpha(\lambda)-\lambda)\lambda_n\in
[K,K]_{\alpha^{mn}}\right\}}{[K,K]_{\alpha^{(m+1)n}}+\ima\left(\sum_{\ell=0}^{n-1}\alpha^{\ell}
\right)}x^{n-1},\\
& \HH^K_{2m+2}(A)= \frac{\left\{\lambda\in K: \sum_{\ell=0}^{n-1}\alpha^{\ell}(\lambda) \in
[K,K]_{\alpha^{(m+1)n}}\right\}} {[K,K]_{\alpha^{(m+1)n}}+\ima(\alpha-\id)\lambda_n}.
\end{align*}
\end{coro}

\begin{rem}\label{remark 2.5} Note that the results in Theorems~\ref{teorema 2.2} and
\ref{teorema 2.3}, and Corollary~\ref{corolario 2.4} do not depend on $f$, with the exception of
its degree $n$ and its independent term $\lambda_n$.
\end{rem}

Assume now that $k$ is a field, the hypothesis of Theorem~\ref{teorema 2.2} are fulfilled and
$\alpha$ is diagonalizable. Let $\omega_1 = 1, \omega_2,\dots,\omega_s$ be the eigenvalues of
$\alpha$ and let $K^{\omega_h}$ be the eigenspace of eigenvalue $\omega_h$. Write
$$
[K,K]^{\omega_h}_{ \alpha^r} = K^{\omega_h}\cap [K,K]_{\alpha^r}.
$$
Note that $1,\lambda_1,\dots,\lambda_n\in K^1$. Moreover, using the properties of
$\breve{\lambda}$ and that $\alpha$ is an algebra endomorphism, it is easy to see that there is a
primitive $n$-th root of $1$ in  $k$ and that all the $n$-th roots of $1$ in $k$ are eigenvalues
of $\alpha$. Consequently, the characteristic of $k$ does not divide~$n$.

\begin{teo}\label{teorema 2.6} The chain complex $C^S(A)$ decomposes as a direct sum $C^S(A) =
\bigoplus_{h=1}^s C^{S,\omega_h}(A)$, where
\[
C_r^{S,\omega_h}(A) = \begin{cases} \frac{K^{\omega_h}}{[K,K]^{\omega_h}_{\alpha^{mn}}}&\text{if
$r=2m$,}\\[8pt]
\frac{K^{\omega_h}}{[K,K]^{\omega_h}_{\alpha^{(m+1)n}}}x^{n-1}& \text{if $r=2m+1$.}\end{cases}
\]
Moreover the boundary maps $d_*^{\omega_h}$ of $C_r^{S,\omega_h}(A)$ are given by:
$$
d_{2m}^{\omega_1}([\lambda]) = \left(\sum_{l=0}^{n-1} \omega_h^l\right)
[\lambda]x^{n-1}\quad\text{and}\quad d_{2m+1}^{\omega_h}([\lambda]x^{n-1}) =
(\omega_h-1)[\lambda\lambda_n].
$$
\end{teo}

\begin{proof} It follows easily from Theorem~\ref{teorema 2.2} and \ref{teorema 2.3}, since
$\lambda_n\in K^1$.
\end{proof}

\begin{coro}\label{corolario 2.7} Let $\HH^{K,\omega_h}_*(A)$ denote the homology of
$C^{S,\omega_h}(A)$. By Theorem~\ref{teorema 2.6} we know that $\HH^K_*(A) = \bigoplus_{h=1}^s
\HH^{K,\omega_h}_*(A)$. Moreover,

\begin{align*}
& \HH^{K,\omega_h}_0(A) = \begin{cases} \frac{K^1}{[K,K]^1} & \text{if $h = 1$,} \\[8pt]
\frac{K^{\omega_h}}{[K,K]^{\omega_h} + K^{\omega_h}\lambda_n} & \text{if $h \ne 1$,}\end{cases}\\[5pt]
& \HH^{K,\omega_h}_{2m+1}(A) = \begin{cases} \frac{\left\{\lambda \in
K^{\omega_h}:\lambda\lambda_n\in
[K,K]^{\omega_h}_{\alpha^{mn}}\right\}}{[K,K]^{\omega_h}_{\alpha^{(m+1)n}}}x^{n-1} & \text{if $h
\ne 1$ and
$\omega_h^n = 1$,}\\[8pt] 0 &\text{otherwise,}\end{cases}\\[5pt]
& \HH^{K,\omega_h}_{2m+2}(A) = \begin{cases}
\frac{K^{\omega_h}}{[K,K]^{\omega_h}_{\alpha^{(m+1)n}}
+K^{\omega_h}\lambda_n} & \text{if $h \ne 1$ and $\omega_h^n = 1$,}\\[8pt] 0 &\text{otherwise,} \end{cases}
\end{align*}

\end{coro}

\noindent Note that if $\alpha^n$ has finite order $v$ (that is $\alpha^{nv} = \id$ and
$\alpha^{nj}\ne \id$ for $0<j<v$), then
$$
\HH^{K,\omega_h}_{2m+1}(A) = \HH^{K,\omega_h}_{2(m+v)+1}(A)\quad\text{and}\quad
\HH^{K,\omega_h}_{2m+2}(A) = \HH^{K,\omega_h}_{2(m+v)+2}(A)
$$
for all $m\ge 0$.

\begin{ex}\label{ejemplo 2.8} Let $k$ be a field, $K = k[G]$ the group $k$-algebra of a finite
group $G$ and $\chi\colon G\to k^{\times}$ a character, where $k^{\times}$ is the group of unities
of $k$. Let $\alpha\colon K\to K$ be the automorphism defined by $\alpha(g) =\chi(g)g$ and let $f
= x^n + \lambda_1 x^{n-1}+\cdots +\lambda_n\in K[x]$ be a monic polynomial whose coefficients
satisfy the hypothesis required in the introduction. Assume that there exists $g_1\in\Z(G)$ such
that $\chi(g_1)$ is a primitive $n$-th root of $1$. In this section we apply the results obtained
in Section~2 to compute the Hochschild homology of $A = K[x,\alpha]/\langle f\rangle$ relative to
$K$ (if the characteristic of $k$ is relative prime to the order of $G$, then $k[G]$ is a
separable $k$-algebra and so, by \cite[Theorem 1.2]{G-S}, $\HH^K_*(A)$ coincides with the absolute
Hochschild homology $\HH_*(A)$ of $A$). Note that the hypothesis of Theorem~\ref{teorema 2.2} are
fulfilled, taking $\breve \lambda = g_1$. In particular the homological behavior of $A$ is
independent of the polynomial $f$, with the exception of its degree and its independent term.
Since $\alpha$ is diagonalizable Theorem~\ref{teorema 2.6} and Corollary~\ref{corolario 2.7}
apply. In this case
\begin{align*}
&\{\omega_1,\dots,\omega_s\} = \chi(G),\\[5pt]
&K^{\omega_h} = \bigoplus_{\{g\in G:\chi(g) = \omega_h\}} kg,\\[5pt]
&[K,K]^{\omega_h}_{\alpha^j} = \sum_{\{g_1,g_2\in G:\chi(g_1g_2) = \omega_h\}}
k(\chi^j(g_2)g_1g_2- g_2g_1).
\end{align*}
As a concrete example take the Dihedral group $D_{2u}$. That is the group generated by $g,h$ and
the relations $g^u = h^2 = 1$ and $hg=g^{-1}h$. Consider the character $\chi\colon D_{2u}\to \C$
defined by $\chi(g^jh^l) = (-1)^l$. Let $A$ be the monogenic extension of $K = \C[D_{2u}]$
associated with $\alpha$ and the polynomial $f = x^2$. Direct computations show that
\begin{align*}
&K^1 = \C[\langle g\rangle],\\
& K^{-1} = \C[\langle g\rangle]h,\\
&[K,K]^1 = \bigoplus_{j=1}^{[(u-1)/2]} \C(g^j-g^{u-j}),\\
& [K,K]^{-1} = \C[\langle g\rangle](g^2-1)h.
\end{align*}
Hence, applying Corollary~\ref{corolario 2.7}, we obtain
\begin{align*}
&\HH_0(A) = \frac{\C[\langle g\rangle]}{\bigoplus_{j=1}^{[(u-1)/2]} \C(g^j-g^{u-j})} \oplus
\frac{\C[\langle g\rangle]h}{\C[\langle g\rangle](g^2-1)h},\\
&\HH_{2m+1}(A) = \frac{\C[\langle g\rangle]h}{\C[\langle g\rangle](g^2-1)h} x,\\
&\HH_{2m+2}(A) = \frac{\C[\langle g\rangle]h}{\C[\langle g\rangle](g^2-1)h}.
\end{align*}

\end{ex}

Next we consider another situation in which the cohomology of $A$ can be compute. The following
results are very closed to the ones valid in the commutative setting.

\begin{teo}\label{teorema 2.9} If $\alpha$ is the identity map, then
$$
C^S_r(A) = \frac{K}{[K,K]}\oplus \frac{K}{[K,K]} x\oplus \cdots \oplus \frac{K}{[K,K]}x^{n-1} =
\frac{A}{[A,A]}.
$$
Moreover, the odd coboundary maps $d_{2m+1}$ of $C^S(A)$ are zero, and the even coboundary maps
$d_{2m}$ takes $[a]$ to $[f'a]$.
\end{teo}

\begin{proof} This is immediate.
\end{proof}

\begin{coro}\label{corolario 2.10} If $\alpha$ is the identity map, then
\begin{align*}
& \H^K_0(A) = \frac{A}{[A,A]},\\
& \H^K_{2m+1}(A) = \frac{A}{[A,A]+f'A} ,\\
& \H^K_{2m+2}(A) = \frac{([A,A]:f')}{[A,A]},
\end{align*}
where $([A,A]:f') = \left\{a\in A: f'a\in [A,A]\right\}$.
\end{coro}

\subsection{Hochschild homology of rank~$1$ Hopf algebras} Let $k$ be a characteristic zero
field. Recall that $k^{\times}$ denotes the group of unities of $k$. Let $G$ be a finite group and
$\chi\colon G\to k^{\times}$ a character. Assume that there exists $g_1\in \Z(G)$ such that
$\chi(g_1)$ is a primitive $n$-th root of $1$. In this section we compute the Hochschild homology
of $A = k[G][x,\alpha]/\langle x^n-\xi(g_1^n-1)\rangle$ over $k$, where $\xi\in k$ and $\alpha\in
\Aut(k [G])$ is defined by $\alpha(g) = \chi(g)g$. We divide the problem in three cases. The first
and second ones give Hochschild homology of rank~$1$ Hopf algebras. Let $C_n\subseteq k$ be the
set of all $n$-th roots of~$1$.

\medskip

\noindent\bf $\boldsymbol{\xi} \boldsymbol{=} \boldsymbol{0}$.\rm\enspace In this case $A =
K[x,\alpha]/\langle x^n\rangle$, where $K = k[G]$. From Corollary~\ref{corolario 2.7} it follows
that
\begin{align*}
&\HH_0(A) = \frac{K}{[K,K]},\\
&\HH_{2m+1}(A) = \bigoplus_{w\in C_n\setminus\{1\}} \frac{K^w}{[K,K]^w_{\alpha^{(m+1)n}}}
x^{n-1},\\
&\HH_{2m+2}(A) = \bigoplus_{w\in C_n\setminus\{1\}} \frac{K^w}{[K,K]^w_{\alpha^{(m+1)n}}}.
\end{align*}

\medskip

\noindent\bf $\boldsymbol{\xi} \boldsymbol{\ne} \boldsymbol{0}$ and
$\boldsymbol{\chi}^{\mathbf{n}} \boldsymbol{=} \boldsymbol{\id}$.\rm\enspace In this case
$f=x^n-\xi(g_1^n-1)$ satisfies the hypothesis required in the introduction (that is
$\alpha(\xi(g_1^n-1)) = \xi(g_1^n-1)$ and $\xi(g_1^n-1)\lambda = \alpha^n(\lambda)\xi(g_1^n-1)$).
Moreover $K = k[G]$ is separable over $k$ and so, $\HH_*(A) = \HH^K_*(A)$. By
Corollary~\ref{corolario 2.7} we have
\begin{align*}
&\HH_0(A) = \frac{K^1}{[K,K]^1}\oplus \bigoplus_{w\in C_n\setminus\{1\}}
\frac{K^w}{[K,K]^w+K^w(g_1^n-1)},\\
&\HH_{2m+1}(A) = \bigoplus_{w\in C_n\setminus\{1\}} \frac{\{\lambda\in K^w:\lambda(g_1^n-1)\in
[K,K]^w \}}{[K,K]^w}x^{n-1},\\
&\HH_{2m+2}(A) = \bigoplus_{w\in C_n\setminus\{1\}} \frac{K^w}{[K,K]^w + K^w(g_1^n-1)}.
\end{align*}

\medskip

\noindent\bf $\boldsymbol{\xi} \boldsymbol{\ne} \boldsymbol{0}$ and
$\boldsymbol{\chi}^{\mathbf{n}} \boldsymbol{\ne} \boldsymbol{\id}$.\rm\enspace Let $g\in G$ such
that $\chi^n(g)\neq 1$. Since
$$
g^{-1}(x^n-\xi(g_1^n-1))g = \chi^n(g)x^n - \xi(g_1^n -1),
$$
we conclude that the ideal $\cl{x^n-\xi(g_1^n-1)}$ coincides with the ideal $\cl{x^n, g_1^n-1}$.
So, $A = k[G/\cl{g_1^n}][x,\wt{\alpha}]/\cl{x^n}$, where $\wt{\alpha}$ is the automorphism induced
by $\alpha$. We consider now $K=k[G/\cl{g_1^n}]$ and $f=x^n$. These data satisfy the hypothesis of
Theorem~\ref{teorema 2.6}. Moreover the algebra $K=k[G/\cl{g_1^n}]$ is separable over $k$ and so,
$\HH_*(A) = \HH^K_*(A)$. Thus, by Corollary~\ref{corolario 2.7}, we have
\begin{align*}
&\HH_0(A) = \frac{K}{[K,K]},\\
&\HH_{2m+1}(A) = \bigoplus_{w\in C_n\setminus\{1\}} \frac{K^w}{[K,K]^w_{\wt{\alpha}^{(m+1)n}}}
x^{n-1},\\
&\HH_{2m+2}(A) = \bigoplus_{w\in C_n\setminus\{1\}} \frac{K^w}{[K,K]^w_{\wt{\alpha}^{(m+1)n}}}.
\end{align*}

\section{Cyclic homology of $A$}
In this section we get a mixed complex, simpler than the canonical of Tsygan computing the cyclic
homology of a noncommutative monogenic extension $A$.

A simple tensor $a_0\ot \cdots \ot a_r\in A\ot \ra{A}^{\ot^r}$ will be called {\em monomial} if
there exist $\lambda\in K\setminus \{0\}$,  $0\le i_0<n$ and  $1\le i_1,\dots,i_r<n$ such that
$a_0=\lambda x^{i_0}$ and $a_j=x^{i_j}$ for $j>0$. We define the {\em degree} of a monomial tensor
$$
\lambda x^{i_0}\ot \cdots \ot x^{i_r}\in A\ot \ra{A}^{\ot^r},
$$
as $\deg(\lambda x^{i_0}\ot \cdots \ot x^{i_r})=i_0+\cdots + i_r$. Since $1,x,\dots,x^{n-1}$ is a
basis of $\ra{A}$ as a left $K$-module, each element $\a \in A\ot \ra{A}^{\ot^r}$ can be written
in a unique way as a sum of monomial tensors. The {\em degree} $\deg(\a)$, of $\a$, is the maximum
of the degrees of its monomial tensors. Since $[A\ot\ra{A}^{\ot^r},K]$ is an homogeneous
$k$-submodule of $A\ot\ra{A}^{\ot^r}$ we have a well defined concept of degree on
$A\ot\ra{A}^{\ot^r}\ot$. Similarly it can be defined the {\em degree} $\deg(\a)$ of an element
$\a\in A\ot \ra{A}^{\ot^r}\ot A$.

\begin{prop}\label{prop 3.1} It is true that $\deg(\omega_{r+1}(\a))\le \deg(\a)$.
\end{prop}

\begin{proof} Let $\x_1 = 1\ot x^{i_1}\ot\cdots \ot x^{i_r}\ot 1\in A\ot\ra{A}^{\ot^r}\ot A$. By
the definition of $w_{r+1}$ it suffices to show that $w'_{r+1}(\x_1)$ is a sum of tensors of the
form
$$
\lambda'x^{j_0}\ot x^{j_1}\ot\cdots\ot x^{j_{r+2}},
$$
with $j_0+ \cdots + j_{r+2} \le i_1+\cdots+i_r$. From the definitions it follows that
$$
\deg(\phi'_r\psi'_r(\x_1)) \le \deg(\x_1).
$$
The proposition follows now by induction on $r$, since
$$
w'_{r+1}(\x_1) = (-1)^r\phi'_r\psi'_r(\x_1)\ot 1 - w'_r(\x_2)x^{i_r}\ot 1,
$$
where  and $\x_2 = 1\ot x^{i_1}\ot\cdots \ot x^{i_{r-1}}\ot 1$.
\end{proof}

\medskip

Let $D_r\colon C^S_r(A)\to C^S_{r+1}(A)$ be the composition $D_r = \phi_{r+1} B_r \psi_r$.

\begin{teo}\label{teorema 3.2} $(C^S_*(A),d_*,D_*)$ is a mixed complex, giving the cyclic homology of $A$.
\end{teo}

\begin{proof} From Theorem~\ref{teorema 2.1} we get a special deformation retract between the
total complexes of the double complexes $\BC(A\ot\ra{A}^{\ot^*}\ot,b,0)$ and $\BC(C^S_*(A),d_*,
0)$. Consider the perturbation $B$. By the perturbation lemma it suffices to prove that
$\psi_{r+2j+1} (B \omega)^j B \phi_r = 0$ for all $j>0$. Assume first that $r = 2m$. By the
definition of $\phi_{2m}$ and Proposition~\ref{prop 3.1},
$$
\deg((B \omega)^j B \phi_{2m}([\lambda' x^j]) < mn+n
$$
On the other hand $\psi_{2m+2j+1}$ vanishes on element of degree lesser or equal of $(m+j)n$.
The assertion follows by combining theses facts. The case $r = 2m+1$ is similar.
\end{proof}

\begin{teo}\label{teorema 3.3} The Connes operator $D_*$  is given by
\begin{align*}
& D_{2m}([\lambda x^j])=\left[\sum_{h=0}^{j-1}\alpha^{mn+h}(\lambda) x^{j-1}\right]\\
& \phantom{D_{2m}([\lambda x^j])}+\sum_{u=0}^{m-1} \left[\ra{\sum_{i=1}^n\lambda_{n-i}
\left(\sum_{l=0}^{i-1} \alpha^{nu+l}(\lambda)\right) x^{j+i-1}}
\right],\\
& D_{2m+1}([\lambda x^j]) = \begin{cases}0&\text{if $j<n-1$,}\\
(\id - \alpha) (\sum_{u=0}^m \alpha^{nu}(\lambda))&\text{if $j=n-1$.}
\end{cases}
\end{align*}
\end{teo}
\begin{proof} Besides the notations introduced in Theorem~\ref{teorema 1.2} we use
the following ones.

\begin{itemize}

\smallskip

\item $\breve{\x}^{\boldsymbol{\ell}_{u,1}} = x^{l_u}\ot x\ot\cdots \ot x^{l_1}\ot x$,

\smallskip

\item $\wt{\x}^{\boldsymbol{\ell}_{m,u+1}} = x\ot x^{l_m}\ot\cdots \ot x\ot x^{l_{u+1}}$,

\smallskip

\item $|\boldsymbol{\ell}_{u,1}| = \ell_1 +\cdots + \ell_u + u$.

\smallskip

\end{itemize}

\noindent We shall first compute $D_{2m+1}$. By definition
$$
B \phi_{2m+1}([\lambda x^j]) = \sum_{u=0}^m \sum_{\mathbf{i}\in\I_m} \sum_{\boldsymbol{\ell}\in
\J_{\mathbf i}} \Delta_{\mathbf{i},u}^{\boldsymbol{\ell}} - \sum_{u=0}^m \sum_{\mathbf{i}\in\I_m}
\sum_{\boldsymbol{\ell}\in \J_{\mathbf i}} \Upsilon_{\mathbf{i}}^{\boldsymbol{\ell}},
$$
where
\begin{align*}
& \Delta_{\mathbf{i},u}^{\boldsymbol{\ell}} = [\boldsymbol{\lambda}_{\mathbf{n} -
\mathbf{i}}\alpha^{|\boldsymbol{\ell}_{u,1}|}(\lambda) \ot \breve{\x}^{\boldsymbol{\ell}_{u,1}}
\ot x^j
x^{|\mathbf{i}- \boldsymbol{\ell}|-m}\ot \wt{\x}^{\boldsymbol{\ell}_{m,u+1}}\ot x]\\
& \Gamma_{\mathbf{i},u}^{\boldsymbol{\ell}} = [\boldsymbol{\lambda}_{\mathbf{n} -
\mathbf{i}}\alpha^{|\boldsymbol{\ell}_{u,1}|+1}(\lambda) \ot \wt{\x}^{\boldsymbol{\ell}_{u,1}}\ot
x \ot x^j x^{|\mathbf{i}- \boldsymbol{\ell}|-m}\ot \wt{\x}^{\boldsymbol{\ell}_{m,u+1}}].
\end{align*}
If $\psi_{2m+2}(\Delta_{\mathbf{i},u}^{\boldsymbol{\ell}}) \ne 0$, then $l_1 = \cdots = l_m =
n-1$. So $i_1 = \cdots = i_m = n$. Thus,
$$
\sum_{\mathbf{i}\in\I_m} \sum_{\boldsymbol{\ell}\in \J_{\mathbf i}}
\psi_{2m+2}(\Delta_{\mathbf{i},u}^{\boldsymbol{\ell}}) = [\alpha^{nu}(\lambda) \ra{x^{j+1}}] =
\begin{cases} 0 & \text{if $j<n-1$,}\\ [\alpha^{nu}(\lambda)] & \text{if $j=n-1$.} \end{cases}
$$
Similarly, $\psi_{2m+2}(\Gamma_{\mathbf{i},u}^{\boldsymbol{\ell}}) \ne 0$ implies that $l_1 =
\cdots = l_m = n-1$. Hence $i_1 = \cdots = i_m = n$ and
$$
\sum_{\mathbf{i}\in\I_m} \sum_{\boldsymbol{\ell}\in \J_{\mathbf i}}
\psi_{2m+2}(\Gamma_{\mathbf{i},u}^{\boldsymbol{\ell}}) = [\alpha^{nu+1}(\lambda) \ra{x^{j+1}}] =
\begin{cases} 0 & \text{if $j<n-1$,}\\ [\alpha^{nu+1}(\lambda)] & \text{if $j=n-1$.} \end{cases}
$$
The formula for $D_{2m+1}$ follows immediately from these facts. We now compute $D_{2m}$. By
definition
$$
B \phi_{2m}([\lambda x^j]) = \sum_{u=0}^{m-1} \sum_{\mathbf{i}\in\I_m} \sum_{\boldsymbol{\ell}\in
\J_{\mathbf i}} (\Gamma_{\mathbf{i},u}^{\boldsymbol{\ell}} +
\Delta_{\mathbf{i},u}^{\boldsymbol{\ell}}) + \sum_{\mathbf{i}\in\I_m} \sum_{\boldsymbol{\ell}\in
\J_{\mathbf i}} \Upsilon_{\mathbf{i}}^{\boldsymbol{\ell}},
$$
where
\begin{align*}
& \Gamma_{\mathbf{i},u}^{\boldsymbol{\ell}} = [\boldsymbol{\lambda}_{\mathbf{n}- \mathbf{i}}
\alpha^{|\boldsymbol{\ell}_{u1}|}(\lambda) \ot \wt{\x}^{\boldsymbol{\ell}_{u1}} \ot x^j
x^{|\mathbf{i}-
\boldsymbol{\ell}|-m}\ot \wt{\x}^{\boldsymbol{\ell}_{m,u+1}}],\\
& \Delta_{\mathbf{i},u}^{\boldsymbol{\ell}} = [\boldsymbol{\lambda}_{\mathbf{n} -
\mathbf{i}}\alpha^{|\boldsymbol{\ell}_{u+1,1}|-1}(\lambda) \ot x^{l_{u+1}}\ot
\wt{\x}^{\boldsymbol{\ell}_{u1}} \ot x^j
x^{|\mathbf{i}- \boldsymbol{\ell}|-m}\ot \wt{\x}^{\boldsymbol{\ell}_{m,u+2}}\ot x]\\
& \Upsilon_{\mathbf{i}}^{\boldsymbol{\ell}} = [\boldsymbol{\lambda}_{\mathbf{n}- \mathbf{i}}
\alpha^{|\boldsymbol{\ell}_{m,1}|}(\lambda) \ot \wt{\x}^{\boldsymbol{\ell}_{m,1}} \ot x^j
x^{|\mathbf{i}- \boldsymbol{\ell}|-m}].
\end{align*}
If $\psi_{2m+1}(\Gamma_{\mathbf{i},u}^{\boldsymbol{\ell}}) \neq 0$, then $l_1 = \cdots =
\widehat{l_{u+1}} =\cdots = l_m = n-1$. In this case $i_1 = \cdots = \widehat{i_{u+1}} =\cdots =
i_m = n$ and
\begin{align*}
\psi_{2m+1}(\Gamma_{\mathbf{i},u}^{\boldsymbol{\ell}}) & = \sum_{h=0}^{l-1} \left[x^{l-h-1}
\lambda_{n-i} \alpha^{nu}(\lambda) \ra{\stackrel{\dots\dots\dots\dots
\dots\dots\dots} {x^{j+i-l-1}}x}x^h\right]\\
& =\sum_{h=0}^{l-1} \left[ \lambda_{n-i} \alpha^{nu+l-h-1}(\lambda)
x^{l-h-1}\ra{\stackrel{\dots\dots\dots\dots \dots\dots\dots} {x^{j+i-l-1}}x}x^h\right]\\
& = \sum_{h=0}^{l-1} \left[\lambda_{n-i} \alpha^{nu+l-h-1}(\lambda) x^{l-1}
\ra{\stackrel{\dots\dots\dots\dots
\dots\dots\dots} {x^{j+i-l-1}}x}\right]\\
& = \sum_{h=0}^{l-1} \left[\lambda_{n-i} \alpha^{nu+l-h-1}(\lambda)x^{l-1}\left(\ra{x^{j+i-l}} -
\ra{x^{j+i-l-1}}x \right) \right].
\end{align*}
So,
\begin{align*}
\sum_{\mathbf{i}\in\I_m} \sum_{\boldsymbol{\ell}\in \J_{\mathbf i}}
\psi_{2m+1}(\Gamma_{\mathbf{i},u}^{\boldsymbol{\ell}}) & = \sum_{i=1}^n
\sum_{l=1}^{i-1}\sum_{h=0}^{l-1}\left[\lambda_{n-i}
\alpha^{nu+l-h-1}(\lambda)x^{l-1} \ra{x^{j+i-l}}\right]\\
&-\sum_{i=1}^n\sum_{l=2}^i\sum_{h=1}^{l-1}\left[\lambda_{n-i}\alpha^{nu+l-h-1}(\lambda)x^{l-1}
\ra{x^{j+i-l}} \right]\\
& = \sum_{i=1}^n\sum_{l=1}^{i-1}\left[\lambda_{n-i}\alpha^{nu+l-1}(\lambda)x^{l-1} \ra{x^{j+i-l}}
\right].
\end{align*}
Similarly, $\psi_{2m+1}(\Delta_{\mathbf{i},u}^{\boldsymbol{\ell}})\ne 0$ implies $l_2 = \cdots =
l_m = n-1$. In this case $i_2 = \cdots = i_m = n$ and
\begin{align*}
\psi_{2m+1}(\Delta_{\mathbf{i},u}^{\boldsymbol{\ell}}) & = \left[\lambda_{n-i}
\alpha^{nu+l}(\lambda) \ra{x^{l}\stackrel{\dots\dots\dots\dots\dots\dots\dots}
{x^{j+i-l-1}}}\right]\\
& = \left[\lambda_{n-i} \alpha^{nu+l}(\lambda)\left( \ra{ x^{j+i-1}} -
x^{l}\ra{x^{j+i-l-1}}\right)\right].
\end{align*}
So,
\begin{align*}
\sum_{\mathbf{i}\in\I_m} \sum_{\boldsymbol{\ell}\in \J_{\mathbf{i}}}
\psi_{2m+1}(\Delta_{\mathbf{i},u}^{\boldsymbol{\ell}}) & = \sum_{i=1}^n \left[\lambda_{n-i}
\left(\sum_{l=1}^{i-1} \alpha^{nu+l}(\lambda)\right) \ra{x^{j+i-1}}\right]\\
& - \sum_{i=1}^n \sum_{l=1}^{i-1}\left[\lambda_{n-i}
\alpha^{nu+l}(\lambda)x^{l}\ra{x^{j+i-l-1}}\right].
\end{align*}
Consequently,
\begin{align*}
\sum_{\mathbf{i}\in\I_m} \sum_{\boldsymbol{\ell}\in \J_{\mathbf i}}
\psi_{2m+1}(\Gamma_{\mathbf{i},u}^{\boldsymbol{\ell}} + \Delta_{\mathbf{i},u}
^{\boldsymbol{\ell}}) & = \sum_{i=1}^n \sum_{l=1}^{i-1}\left[\lambda_{n-i}
\alpha^{nu+l-1}(\lambda)x^{l-1}\ra{x^{j+i-l}}\right]\\
& - \sum_{i=1}^n \sum_{l=2}^{i}\left[\lambda_{n-i}
\alpha^{nu+l-1}(\lambda)x^{l-1}\ra{x^{j+i-l}}\right]\\
& + \sum_{i=1}^n \left[\lambda_{n-i} \left(\sum_{l=1}^{i-1}
\alpha^{nu+l}(\lambda)\right)\ra{x^{j+i-1}}\right] \\
& = \sum_{i=1}^n \left[\lambda_{n-i} \alpha^{nu}(\lambda) \ra{x^{j+i-1}}\right]\\
& + \left[\sum_{i=1}^n\lambda_{n-i}\left(\sum_{l=1}^{i-1}
\alpha^{nu+l}(\lambda)\right) \ra{x^{j+i-1}} \right]\\
& = \left[\ra{\sum_{i=1}^n\lambda_{n-i} \left(\sum_{l=0}^{i-1} \alpha^{nu+l}(\lambda)\right)
x^{j+i-1}} \right].
\end{align*}
Lastly, $\psi_{2m+1}(\Upsilon_{\mathbf{i}}^{\boldsymbol{\ell}}) = 0$ except if $l_1 = \cdots =l_m
= n-1$. In this last case $i_1 = \cdots = i_m = n$. So
$$
\sum_{\mathbf{i}\in\I_m} \sum_{\boldsymbol{\ell}\in \J_{\mathbf i}}
\psi_{2m+1}(\Upsilon_{\mathbf{i}}^{\boldsymbol{\ell}}) = \sum_{h=0}^{j-1}
\left[x^{j-h-1}\alpha^{mn}(\lambda) x^h\right] = \left[\sum_{h=0}^{j-1}\alpha^{mn+h}(\lambda)
x^{j-1}\right].
$$
the formula for $D_{2m}$ follows easily from all these facts.
\end{proof}

\subsection{Explicit computations}
In this subsection we compute the cy\-clic homology of $A$ with coefficients in $A$, under
suitable hypothesis. We will freely use the notations introduced at the beggining of Section~2 and
bellow remark~\ref{remark 2.5}. Recall that by Theorem~\ref{teorema 2.2}, if there exists
$\breve{\lambda}\in \Z(K)$ such that
\begin{itemize}

\smallskip

\item $\alpha^n(\breve{\lambda})=\breve{\lambda}$,

\smallskip

\item $\breve{\lambda}-\alpha^i(\breve{\lambda})$ is invertible in $K$ for $1\le i<n$,

\end{itemize}
then
\[
C^S_r(A) = \begin{cases} \frac{K}{[K,K]_{\alpha^{mn}}}&\text{if $r=2m$,}\\
\frac{K}{[K,K]_{\alpha^{(m+1)n}}}x^{n-1}& \text{if $r=2m+1$.}\end{cases}
\]
Moreover, by Theorem~\ref{teorema 2.3}, the Hochschild boundary maps of the mixed complex
$(C^S_*(A),d_*,D_*)$ are given by
\begin{align*}
& d_{2m+1}([\lambda]x^{n-1} )= [(\alpha(\lambda)-\lambda)\lambda_n],\\
& d_{2m+2}([\lambda])= \Biggl[\sum_{\ell=0}^{n-1} \alpha^{\ell}(\lambda) \Biggr]x^{n-1}.
\end{align*}
We now compute the Connes operator $D_*$.

\begin{teo}\label{teorema 3.4} Under the hypothesis of Theorem~\ref{teorema 2.2}, we have:
\begin{align*}
& D_{2m}([\lambda])= 0,\\
& D_{2m+1}([\lambda] x^{n-1}) = \left[(\id - \alpha) \left(\sum_{u=0}^m
\alpha^{nu}(\lambda)\right)\right].
\end{align*}
\end{teo}

\begin{proof} If follows immediately from Theorem~\ref{teorema 3.3}.
\end{proof}

\begin{teo}\label{teorema 3.5} Assume the hypothesis of Theorem~\ref{teorema 2.6} are fulfilled.
Then the mixed complex $(C^S_*(A),d_*,D_*)$ decomposes as a direct sum
$$
(C^S_*(A),d_*,D_*) = \bigoplus_{h=1}^s (C^{S,\omega_h}_*(A),d^{\omega_h}_*,D^{\omega_h}_*),
$$
where the Hochschid complexes $(C^{S,\omega_h}_*(A),d^{\omega_h}_*)$ are as in
Theorem~\ref{teorema 2.6}. Moreover the Connes operators $D^{\omega_h}_*$ satisfy:
$D^{\omega_h}_{2m}=0$, and
$$
D^{\omega_h}_{2m+1}([\lambda]x^{n-1})= (1-\omega_h)\left(\sum_{u=0}^m
\omega_h^{nu}\right)[\lambda].
$$
\end{teo}

\begin{proof} If follows immediately from Theorem~\ref{teorema 3.4}.
\end{proof}

In the rest of this section we assume that $k$ is a characteristic zero field and that hypothesis
of Theorem~\ref{teorema 2.6} are fulfilled. We let $\HC^{K,\omega_h}_*(A)$ denote the cyclic
homology of $(C^{S,\omega_h}_*(A),d^{\omega_h}_*,D^{\omega_h}_*)$.

\begin{teo}\label{teorema 3.6} The cyclic homology of $A$ decomposes as
$$
\HC^K_*(A) = \bigoplus_{h=1}^s \HC^{K,\omega_h}_*(A).
$$
Moreover, we have:
\begin{align*}
& \HC^{K,\omega_h}_{2m}(A) = \begin{cases} \frac{K^1}{[K,K]^1} & \text{if $h=1$,}\\[5pt]
\frac{K^{\omega_h}}{[K,K]^{\omega_h} + K^{\omega_h}\lambda_n} & \text{if $\omega_h^n \ne 1$,}\\[5pt]
\frac{K^{\omega_h}}{[K,K]^{\omega_h}+K^{\omega_h}\lambda_n^{m+1}}& \text{otherwise,}
\end{cases}\\[5pt]
& \HC^{K,\omega_h}_{2m+1}(A) = \begin{cases} 0 & \text{if $h=1$ or $\omega_h^n \ne 1$,}\\[5pt]
\frac{\{\lambda\in K^{\omega_h}: \lambda\lambda_n^m\in [K,K]^{\omega_h}\}}
{[K,K]^{\omega_h}_{\alpha^{(m+1)n}}}x^{n-1} & \text{otherwise,}
\end{cases}
\end{align*}
\end{teo}

\begin{proof} The first assertion is an immediate consequence of Theorem~\ref{teorema 3.5}, and the
computation of $\HC^{K,\omega_h}_*$ for $h=1$ and for $\omega_h^n \ne 1$ follows from
Corollary~\ref{corolario 2.7}. So, in order to finish the proof it remains to consider the case
$h>1$ and $\omega_h^n = 1$. By Theorems~\ref{teorema 2.6} and \ref{teorema 3.4}, the cyclic
homology of the mixed complex $(C^{S, \omega_h}_*(A),d^{\omega_h}_*,D^{\omega_h}_*)$, is the
homology of
\[
\spreaddiagramcolumns{1.75pc} \xymatrix{ \vdots \dto^-{d^{\omega_h}} &\vdots \dto^-{0}& \vdots
\dto^-{d^{\omega_h}}& \vdots \dto^-{0}& \vdots \dto^-{d^{\omega_h}}\\
X_4 \dto^-{0} & X_3 \lto_-{l_{2(1-\omega_h)}}\dto^-{d^{\omega_h}} & X_2 \dto^-{0} \lto_-{0}& X_1
\dto^-{d^{\omega_h}}\lto_-{l_{1-\omega_h}}& X_0 \lto_-{0}\\
X_3 \dto^-{d^{\omega_h}} & X_2 \lto_-{0}\dto^-{0} & X_1 \dto^-{d^{\omega_h}} \lto_-{l_{1-\omega_h}}& X_0 \lto_-{0}\\
X_2 \dto^-{0} & X_1 \lto_-{l_{1-\omega_h}}\dto^-{d^{\omega_h}} & X_0 \lto_-{0}\\
X_1 \dto^-{d^{\omega_h}} & X_0 \lto_-{0}\\
X_0,}
\]
where
\begin{itemize}

\item $X_{2m} = \frac{K^{\omega_h}} {[K,K]^{\omega_h}_{\alpha^{mn}}}$ and $X_{2m+1} = \frac{K^{\omega_h}}
{[K,K]^{\omega_h}_{\alpha^{(m+1)n}}}x^{n-1}$,

\smallskip

\item $l_{m(1-\omega_h)}([\lambda]) = m(1-\omega_h)[\lambda]$,

\smallskip

\item $d_{2m+1}^{\omega_h}([\lambda]x^{n-1}) = (\omega_h-1)[\lambda\lambda_n]$.

\medskip

\end{itemize}
We first compute the homology in degree $2m$. Let
$$
\iota\colon X_0\to X_{2m}\oplus X_{2m-2}\oplus\cdots\oplus X_0
$$
be the canonical inclusion. By using that each $l_{i(1-\omega_h)}$ map is an isomorphism it is
easy to see that $\iota$ induces an epimorphism
$$
\ra{\iota}\colon X_0\to \HC^{K,\omega_h}_{2m}(A).
$$
It is easy to see now that the boundary of
$$
([\zeta_{2m+1}],[\zeta_{(2m-1)+1}],\dots,[\zeta_1])\in X_{2m+1}\oplus X_{(2m-1)+1} \oplus \cdots
\oplus X_1
$$
equals $i([\lambda])$ if and only if
\begin{equation}
[\zeta_{2i+1}] = \frac{(-1)^{m-i}i!}{m!} [\zeta_{2m+1}\lambda_n^{m-i}]\quad\text{for $0\le i\le
m$}\label{eq3}
\end{equation}
and $[\zeta_{2m+1}\lambda_n^{m+1}] = [\lambda]$. The assertion about $\HC_{2m}^{K,\omega_h}(A)$
follows easily from these facts. We now are going to compute the homology in degree $2m+1$. It is
immediate that
$$
([\zeta_{2m+1}],[\zeta_{(2m-1)+1}],\dots,[\zeta_1])\in X_{2m+1}\oplus X_{(2m-1)+1} \oplus \cdots
\oplus X_1
$$
is a cycle of degree $2m+1$ if and only if it satisfies \eqref{eq3} and
$$
\zeta_{2m+1} \lambda_n^{m+1} \in [K,K]^{\omega_h}.
$$
This ends the computation of $\HC_{2m+1}^{K,\omega_h}(A)$.
\end{proof}

\begin{rem}\label{nota 3.7} From the above computations it follows that:

\begin{enumerate}

\item If $h=1$ or $\omega_h^n\ne 1$, then the map
$$
S\colon \HC^{K,\omega_h}_{2m+2}(A)\to \HC^{K,\omega_h}_{2m}(A)
$$
is the identity map.

\smallskip

\item If $h>1$ and $\omega_h^n = 1$, then we have:

\renewcommand{\descriptionlabel}[1]
{\textrm{#1}}
\begin{description}

\smallskip

\item[\rm a.] The map $S\colon \HC^{K,\omega_h}_{2m+2}(A)\to \HC^{K,\omega_h}_{2m}(A)$ is the canonical
surjection.

\smallskip

\item[\rm b.] The map $i\colon \HH^{K,\omega_h}_{2m}(A)\to \HC^{K,\omega_h}_{2m}(A)$ is given by
$$
i([\lambda]) = \frac{1}{m!}[\lambda\lambda_n^m].
$$

\smallskip

\item[\rm c.] The map $B\colon \HC^{K,\omega_h}_{2m}(A)\to \HH^{K,\omega_h}_{2m+1}(A)$ is zero.

\smallskip

\item[\rm d.] The map $S\colon \HC^{K,\omega_h}_{2m+3}(A)\to \HC^{K,\omega_h}_{2m+1}(A)$ is given by
$$
S([\lambda]x^{n-1}) = \frac{1}{m+1}[\lambda\lambda_n]x^{n-1}.
$$

\smallskip

\item[\rm e.] The map $i\colon \HH^{K,\omega_h}_{2m+1}(A)\to \HC^{K,\omega_h}_{2m+1}(A)$ is the canonical
inclusion.

\smallskip

\item[\rm f.] The map $B\colon \HC^{K,\omega_h}_{2m+1}(A)\to \HH^{K,\omega_h}_{2m+2}(A)$ is given by
$$
B([\lambda]x^{n-1}) = [(m+1)(1-\omega_h))\lambda]x^{n-1}.
$$
\end{description}

\end{enumerate}

\end{rem}

\begin{rem}\label{nota 3.8} Theorem~\ref{teorema 3.6} applies in particular to the monogenic
extensions of finite group algebras $K = k[G]$ considered in Example~\ref{ejemplo 2.8}. Note that
since $K$ is a separable $k$-algebra, this computes the absolute cyclic homology, as follows
easily from \cite[Theorem~1.2]{G-S} using the SBI-sequence. We now consider a concrete example.
Let $G=D_{2n}$, $\chi$, $\alpha$ and $A$ be as in Example~\ref{ejemplo 2.8}. Then:
\begin{align*}
&\HC_{2m}(A) = \frac{\C[\langle g\rangle]}{\bigoplus_{j=1}^{[(u-1)/2]} \C(g^j-g^{u-j})} \oplus
\frac{\C[\langle g\rangle]h}{\C[\langle g\rangle](g^2-1)h},\\
&\HC_{2m+1}(A) = \frac{\C[\langle g\rangle]h}{\C[\langle g\rangle](g^2-1)h} x.
\end{align*}

\end{rem}

\subsection{Cyclic homology of rank~$1$ Hopf algebras} Let $k$, $G$, $\chi$, $g_1$, $\alpha$ and
$A$ be as in Subsection~2.2. Here we compute the cyclic homology of $A$. Let $C_n\subseteq k$ be
the set of all $n$-th roots of $1$. As in the above mentioned subsection we consider three cases.

\medskip

\noindent\bf $\boldsymbol{\xi} \boldsymbol{=} \boldsymbol{0}$.\rm\enspace That is  $A =
K[x,\alpha]/\langle x^n\rangle$, where $K = k[G]$. From Theorem~\ref{teorema 3.6} it follows that
\begin{align*}
&\HC_{2m}(A) = \frac{K}{[K,K]},\\
&\HC_{2m+1}(A) = \bigoplus_{w\in C_n\setminus\{1\}} \frac{K^w}{[K,K]^w_{\alpha^{(m+1)n}}} x^{n-1}.
\end{align*}

\medskip

\noindent\bf $\boldsymbol{\xi} \boldsymbol{\ne} \boldsymbol{0}$ and
$\boldsymbol{\chi}^{\mathbf{n}} \boldsymbol{=} \boldsymbol{\id}$.\rm\enspace In this case $A =
K[x,\alpha]/\langle x^n-\xi(g_1^n-1)\rangle$, where $K = k[G]$. Arguing as in Subsection~2.2, but
using Theorem~\ref{teorema 3.6} instead of Corollary~\ref{corolario 2.7}, we obtain
\begin{align*}
&\HC_{2m}(A) = \frac{K^1}{[K,K]^1}\oplus \bigoplus_{w\in C_n\setminus\{1\}}
\frac{K^w}{[K,K]^w+K^w(g_1^n-1)^{m+1}},\\
&\HC_{2m+1}(A) = \bigoplus_{w\in C_n\setminus\{1\}} \frac{\{\lambda\in K^w:\lambda(g_1^n-1)^m\in
[K,K]^w \}}{[K,K]^w}x^{n-1}.
\end{align*}

\medskip

\noindent\bf $\boldsymbol{\xi} \boldsymbol{\ne} \boldsymbol{0}$ and
$\boldsymbol{\chi}^{\mathbf{n}} \boldsymbol{\ne} \boldsymbol{\id}$.\rm\enspace In this case $A =
K[x,\wt{\alpha}]\langle x^n\rangle$, where the algebra $K = k[G/\langle g_1^n\rangle]$ and
$\wt{\alpha}$ is the automorphism induced by $\alpha$. By Theorem~\ref{teorema 3.6}, we obtain
\begin{align*}
&\HC_{2m}(A) = \frac{K}{[K,K]},\\
&\HC_{2m+1}(A) = \bigoplus_{w\in C_n\setminus\{1\}} \frac{K^w}{[K,K]^w_{\wt{\alpha}^{(m+1)n}}}
x^{n-1}.
\end{align*}


\begin{thebibliography}{Bach}


\bibitem[Bach]{Bach} Jorge A. Guccione, Juan J. Guccione, Mar\'\i a J. Redondo, Andrea Solotar y
Orlando E. Villamayor,
\newblock {\em Cyclic Homology of Monogenic Algebras,}
\newblock {\rm Communications in Algebra 22 (12), 4899-4904}
\newblock {(1994).}

\bibitem[B]{B} D. Burghelea,
\newblock {\em Cyclic homology and algebraic $K$-theory of spaces I},
\newblock {Boulder Colorado 1983, Contemp. Math.~55, 89--115}
\newblock {(1986)}.

\bibitem[C]{C} M. Crainic,
\newblock {\em On the perturbation lemma, and deformations},
\newblock {arXiv:Math. AT/0403266}
\newblock {(2004)}.

\bibitem[F-G-G]{F-G-G} M. Farinati, J. A. Guccione and J. J. Guccione,
\newblock {\em The cohomology of monogenic extensions in
the noncommutative setting,}
\newblock {Preprint,}
\newblock {(2007).}

\bibitem[G-S]{G-S} M. Gerstenhaber and S. D. Schack,
\newblock {\em Relative Hochschild cohomology, rigid algebras, and the Bockstein,}
\newblock {J. Pure Appl. Algebra 43, no. 1, 53--74,}
\newblock {(1986).}

\bibitem[G-G]{G-G} J. A. Guccione and J. J. Guccione,
\newblock {\em Hochschild (co)homology of Hopf crossed products,}
\newblock {K-theory 25, 139--169,}
\newblock {(2002).}

\bibitem[K]{K} K. Kassel,
\newblock {\em Cyclic homology, comodules and mixed complexes},
\newblock {Journal of Algebra~107, 195--216}
\newblock {(1987)}.

\bibitem[K-R]{K-R} L. Krop and D. Radford,
\newblock {\em Finite dimensional Hopf algebras of rank 1 in characteristic 0},
\newblock {Journal of Algebra 302, no. 1, 214-230}
\newblock {(2006).}

\end{thebibliography}
\end{document}